\documentclass[12pt]{article}

\usepackage{amsfonts}
\usepackage{amsmath}
\usepackage{amssymb}
\usepackage{amsthm}
\usepackage{verbatim}
\usepackage{hyperref}
\usepackage{tikz}
\usetikzlibrary{calc}

\hypersetup{
  colorlinks,
  urlcolor=black,
  linkcolor=black,
  anchorcolor=black,
  citecolor=black,
}

\newcommand{\bin}[2]{\binom{#1}{#2}}

\newcommand{\nat}{\mathbb{N}}

\newtheorem{theorem}{Theorem}[section]

\newtheorem{lemma}[theorem]{Lemma}
\newtheorem{corollary}[theorem]{Corollary}
\newtheorem{claim}[theorem]{Claim}

\newtheorem{conjecture}[theorem]{Conjecture}
\newtheorem{fact}[theorem]{Fact}
\newcommand{\lp}{\left(}
\newcommand{\rp}{\right)}
\newcommand{\ca}[1]{{\cal #1}}
\newcommand{\K}{{\cal K}}
\newcommand{\N}[1]{\ca{ND}(#1)}
\newcommand{\D}{\cal D}

\title{Complete $r$-partite graphs determined by their domination polynomial}

\author{Barbara M. Anthony\thanks{Department of Mathematics and Computer Science, Southwestern University, Georgetown, TX, USA. E-mail: {\tt anthonyb@southwestern.edu}} \and Michael E. Picollelli\thanks{Department of Mathematical Sciences, Carnegie Mellon University, Pittsburgh, PA, USA. E-mail: {\tt mpicolle@andrew.cmu.edu}}}

\date{}

\begin{document}

\maketitle

\begin{abstract}
The domination polynomial of a graph is the polynomial whose coefficients count the number of dominating sets of each cardinality.  A recent question asks which graphs are uniquely determined (up to isomorphism) by their domination polynomial.  In this paper, we completely describe the complete $r$-partite graphs which are; in the bipartite case, this settles in the affirmative a conjecture of Aalipour, Akbari and Ebrahimi~\cite{AAE}.
\end{abstract}

\section{Introduction}
\label{sec:intro}
While some graph-invariant polynomials are well-studied (including the chromatic polynomial for a century~\cite{Birkhoff}, the independence number polynomial for decades~\cite{GH}), other graph-invariant problems have not benefited from years of exploration. Recently there has been growing interest in the domination polynomial of a graph (see, for instance \cite{AAE, AAP, AP, AL}). We consider in particular a research problem posed at the 22nd British Combinatorial Conference (problem 519 in \cite{BCC}) in July 2009. We solve that problem, proving the conjecture of Aalipour, Akbari and Ebrahimi~\cite{AAE} in the affirmative, and provide some additional results about domination polynomials.

We begin with some notation. For a simple graph $G=(V,E)$, we let $v(G)=|V|$ and $e(G)=|E|$ denote the order and size, respectively, of $G$. For each vertex $v \in V$, we let $N_G(v)$ and $d_G(v)=|N_G(v)|$ denote the neighborhood and degree, respectively, of $v$ in $G$, and $N_G[v]=N_G(v)\cup \{v\}$ denote the closed neighborhood of $v$ in $G$.  We drop the subscript when the graph $G$ is clear from the context.

A set of vertices $S \subseteq V$ is said to be a {\em dominating set}, or that $S$ {\em dominates} $G$, if $\bigcup_{x \in S} N[x]=V$, i.e., every vertex in $V \setminus S$ has a neighbor in $S$.  Similarly, we call a set of vertices $T \subseteq V$ which does not dominate a {\em non-dominating set}, or say that $T$ does not dominate $G$. It is immediate that a subset of a non-dominating set is also a non-dominating set.

Let $d(G,i)$ denote the number of dominating sets in $G$ of cardinality $i$. Note that for non-empty $G$, $d(G,0)=0$.  The {\em domination polynomial} of $G$, introduced in~\cite{AP}, is the polynomial \[D(G,x)=\sum_{i=1}^{v(G)}d(G,i)\cdot x^i.\]
This polynomial is easily seen to be a graph invariant and subsequently induces a natural equivalence relation on graphs.  Two graphs $G$ and $H$ are deemed to be {\em $\D$-equivalent} if $D(G,x)=D(H,x)$.  Letting $[G]$ denote the equivalence class of $G$ up to isomorphism under this relation, we say $G$ is {\em $\D$-unique} if $[G]=\{G\}$, i.e., if $D(H,x)=D(G,x)$ implies $H \cong G$.

A question of recent interest concerning this equivalence relation $[\cdot]$ asks which graphs are determined by their domination polynomial.  It is known that cycles~\cite{AAP},\cite{APCycles} and cubic graphs of order 10~\cite{APOrder10} (particularly, the Petersen graph) are, while if $n \equiv 0 (\mod 3)$, the paths of order $n$ are not~\cite{AAP}.
However, the question remains open for most graphs. Aalipour-Hafshejani, Akbari and Ebrahimi~\cite{AAE} considered this question for the complete bipartite graph, which we will denote by $\K(a,b)$ to allow for extension to the multipartite case.

To state their results, we introduce some further notation. We let $\nat$ denote the positive integers.  Let $\chi(G)$ be the chromatic number of a graph $G$. For $t \in \{0,1\}$ and $a \in \nat$, let $H_t(a)$ denote the graph on $2a+t$ vertices formed by a copy of $K_a$ and a copy of $K_{a+t}$ connected by a matching from $K_a$ into $K_{a+t}$. An example is given in Figure \ref{fig:hta}. Equivalently, $H_0(a)$ is the Cartesian graph product $K_a \times K_2$, while $H_1(a)$ is the graph formed by deleting a vertex from $H_0(a+1)$.  We point out that $H_0(1)\cong \K(1,1)$, $H_0(2)=\K(2,2)$, and $H_1(1)\cong \K(1,2)$, but, in general, $\chi(H_t(a))=a+t$ and hence $H_t(a)\not\cong\K(a,a+t)$ for all other $a,t$.

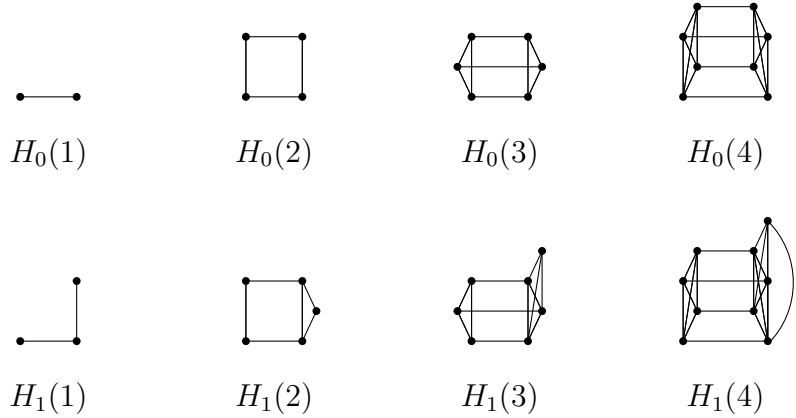
\begin{figure}
\label{fig:hta}

\begin{center}
\begin{tikzpicture}
\foreach \wid/\hgt in {.75/.4}{

% The graph H_0(1)

\foreach \x/\y in {0/1.25}{

    \foreach \a/\er in {0/0}{
        % Fill in the associated vertex on the left, then the right
        \fill (\x - \er*\wid ,\y + \hgt*\a) circle (1.5pt);
        \fill (\x + \wid + \er*\wid, \y + \hgt*\a) circle (1.5pt);

        % Then draw the matching
        \draw (\x - \er*\wid ,\y + \hgt*\a) -- (\x + \wid + \er*\wid, \y + \hgt*\a);
    }

    % Label the image:
    \node [below] at (\x+.5*\wid,\y - \hgt) {$H_0(1)$};
}

% The graph H_0(2):

\foreach \x/\y in {3/1.25}{

    \foreach \a/\er in {0/0,2/0}{

        % Fill in the associated vertex on the left, then the right

        \fill (\x - \er*\wid ,\y + \hgt*\a) circle (1.5pt);
        \fill (\x + \wid + \er*\wid, \y + \hgt*\a) circle (1.5pt);

        % Then draw the matching
        \draw (\x - \er*\wid ,\y + \hgt*\a) -- (\x + \wid + \er*\wid, \y + \hgt*\a);

        % Then connect it to the others:
        \foreach \b/\ers in {0/0,2/0}{
            \draw (\x - \er*\wid ,\y + \hgt*\a) -- (\x - \ers*\wid ,\y + \hgt*\b);
            \draw (\x + \wid + \er*\wid ,\y + \hgt*\a) -- (\x +\wid + \ers*\wid ,\y + \hgt*\b);
        }

    }

    % Label the image:
    \node [below] at (\x+.5*\wid,\y - \hgt) {$H_0(2)$};

}

% The graph H_0(3):

\foreach \x/\y in {6/1.25}{
    \foreach \a/\er in {0/0,1/.25,2/0}{
        % Fill in vertices on the left then right

        \fill (\x - \er*\wid ,\y + \hgt*\a) circle (1.5pt);
        \fill (\x + \wid + \er*\wid, \y + \hgt*\a) circle (1.5pt);

        % Then draw the matching
        \draw (\x - \er*\wid ,\y + \hgt*\a) -- (\x + \wid + \er*\wid, \y + \hgt*\a);

        % Now, draw the edges in each part:

        \foreach \b/\ers in {0/0,1/.25,2/0}{
            \draw (\x - \er*\wid ,\y + \hgt*\a) -- (\x - \ers*\wid ,\y + \hgt*\b);
            \draw (\x + \wid + \er*\wid ,\y + \hgt*\a) -- (\x +\wid + \ers*\wid ,\y + \hgt*\b);
        }
    }

    % Label the image:
    \node [below] at (\x+.5*\wid,\y - \hgt) {$H_0(3)$};
}

% The graph H_0(4):

\foreach \x/\y in {9/1.25}{

    \foreach \a/\er in {0/.25,1/0,2/.25,3/0}{

        % Fill in the associated vertex on the left, then the right

        \fill (\x - \er*\wid ,\y + \hgt*\a) circle (1.5pt);
        \fill (\x + \wid + \er*\wid, \y + \hgt*\a) circle (1.5pt);

      % Then draw the matching
        \draw (\x - \er*\wid ,\y + \hgt*\a) -- (\x + \wid + \er*\wid, \y + \hgt*\a);

     % Then connect it to the others:
        \foreach \b/\ers in {0/.25,1/0,2/.25,3/0}{
            \draw (\x - \er*\wid ,\y + \hgt*\a) -- (\x - \ers*\wid ,\y + \hgt*\b);
            \draw (\x + \wid + \er*\wid ,\y + \hgt*\a) -- (\x +\wid + \ers*\wid ,\y + \hgt*\b);
        }
    }

    % Label the image:
    \node [below] at (\x+.5*\wid,\y - \hgt) {$H_0(4)$};
}

% The graph H_1(1)

\foreach \x/\y in {0/-2.0}{
    \foreach \a/\er in {0/0}{
        % Fill in the associated vertex on the left, then the right
        \fill (\x - \er*\wid ,\y + \hgt*\a) circle (1.5pt);
        \fill (\x + \wid + \er*\wid, \y + \hgt*\a) circle (1.5pt);

        % Then draw the matching
        \draw (\x - \er*\wid ,\y + \hgt*\a) -- (\x + \wid + \er*\wid, \y + \hgt*\a);

    }

    % Extra vertex
    \draw (\x+\wid,\y+2*\hgt) -- (\x+\wid,\y);
    \fill (\x+\wid,\y+2*\hgt) circle (1.5pt);

    % Label the image:
    \node [below] at (\x+.5*\wid,\y - \hgt) {$H_1(1)$};

}

% The graph H_1(2):

\foreach \x/\y in {3/-2.0}{
    \foreach \a/\er in {0/0,2/0}{

        % Fill in the associated vertex on the left, then the right

        \fill (\x - \er*\wid ,\y + \hgt*\a) circle (1.5pt);
        \fill (\x + \wid + \er*\wid, \y + \hgt*\a) circle (1.5pt);

        % Then draw the matching
        \draw (\x - \er*\wid ,\y + \hgt*\a) -- (\x + \wid + \er*\wid, \y + \hgt*\a);

        % Then connect it to the others:
        \foreach \b/\ers in {0/0,2/0}{
            \draw (\x - \er*\wid ,\y + \hgt*\a) -- (\x - \ers*\wid ,\y + \hgt*\b);
            \draw (\x + \wid + \er*\wid ,\y + \hgt*\a) -- (\x +\wid + \ers*\wid ,\y + \hgt*\b);
        }

    }

    % Extra Vertex

    \fill (\x + 1.25*\wid, \y + \hgt) circle (1.5pt);

    \foreach \a in {0,2}
        \draw (\x + 1.25*\wid, \y + \hgt) -- (\x + \wid,\y + \a*\hgt);

    % Label the image:
    \node [below] at (\x+.5*\wid,\y - \hgt) {$H_1(2)$};

}

% The graph H_1(3):

\foreach \x/\y in {6/-2.0}{

    \foreach \a/\er in {0/0,1/.25,2/0}{

        % Fill in the associated vertex on the left, then the right

        \fill (\x - \er*\wid ,\y + \hgt*\a) circle (1.5pt);
        \fill (\x + \wid + \er*\wid, \y + \hgt*\a) circle (1.5pt);

        % Then draw the matching
        \draw (\x - \er*\wid ,\y + \hgt*\a) -- (\x + \wid + \er*\wid, \y + \hgt*\a);

        % Then connect it to the others:
        \foreach \b/\ers in {0/0,1/.25,2/0}{
            \draw (\x - \er*\wid ,\y + \hgt*\a) -- (\x - \ers*\wid ,\y + \hgt*\b);
            \draw (\x + \wid + \er*\wid ,\y + \hgt*\a) -- (\x +\wid + \ers*\wid ,\y + \hgt*\b);
        }
    }

    %Extra Vertex
    \fill (\x + 1.25*\wid, \y + 3*\hgt) circle (1.5pt);

    % Connect it to the others on the right
    \foreach \a/\er in {0/0,1/.25,2/0}
    \draw (\x + 1.25*\wid, \y + 3*\hgt) -- (\x+ \wid + \er*\wid, \y + \a*\hgt);

    % Label the image:
    \node [below] at (\x+.5*\wid,\y - \hgt) {$H_1(3)$};

}

% The graph H_1(4):

\foreach \x/\y in {9/-2.0}{
    \foreach \a/\er in {0/.25,1/0,2/.25,3/0}{
        % Fill in the associated vertex on the left, then the right

        \fill (\x - \er*\wid ,\y + \hgt*\a) circle (1.5pt);
        \fill (\x + \wid + \er*\wid, \y + \hgt*\a) circle (1.5pt);

        % Then draw the matching
        \draw (\x - \er*\wid ,\y + \hgt*\a) -- (\x + \wid + \er*\wid, \y + \hgt*\a);

        % Then connect it to the others:
        \foreach \b/\ers in {0/.25,1/0,2/.25,3/0}{
            \draw (\x - \er*\wid ,\y + \hgt*\a) -- (\x - \ers*\wid ,\y + \hgt*\b);
            \draw (\x + \wid + \er*\wid ,\y + \hgt*\a) -- (\x +\wid + \ers*\wid ,\y + \hgt*\b);
        }
    }

    %Extra Vertex
    \fill (\x + 1.25*\wid, \y + 4*\hgt) circle (1.5pt);
    % Connect it to the others on the right:
    \foreach \a/\er in  {0/.25,1/0,2/.25,3/0}
        \draw (\x + 1.25*\wid, \y + 4*\hgt) -- (\x+ \wid + \er*\wid, \y + \a*\hgt);

    \draw (\x + 1.25*\wid, \y + 4*\hgt) arc (45:-45:2.9*\hgt);

    % Label the image:
    \node [below] at (\x+.5*\wid,\y - \hgt) {$H_1(4)$};
}

}

\end{tikzpicture}
\end{center}

\caption{The graphs $H_0(a)$ and $H_1(a)$ for $1 \le a \le 4$. }
\end{figure}

Aalipour-Hafshejani {\em et al.} observed the following fact and showed the following result.

\begin{fact}[\cite{AAE}]\label{fact:AAE} For any $a \ge 1$ and $t \in \{0,1\}$, $D(\K(a,a+t),x)=D(H_t(a),x)$.
\end{fact}

\begin{theorem}[\cite{AAE}]\label{thm:AAE}
For all $a \in \nat$, $[\K(a,a)] = \{\K(a,a), H_0(a)\}$.
\end{theorem}
They also conjectured that $[\K(a,a+1)] = \{\K(a,a+1), H_1(a)\}$, and established partial results supporting this conjecture.  However, the absence of a similar construction when the size of the partitions differs by more than one led them to the following conjecture.

\begin{conjecture}[\cite{AAE}, Conjecture 2]\label{conj:AAE}
For all $a,b \in \nat$, if $|a-b| \ge 2$ then $\K(a,b)$ is $\ca{D}$-unique.
\end{conjecture}

Some partial progress on Conjecture \ref{conj:AAE} was made in \cite{AAE2}, where it was verified in the case $a \le 4$ or $b \ge \max\{\bin{a}{2},a+2\}$.  Our main result resolves Conjecture \ref{conj:AAE} in the affirmative, and, moreover, provides a complete classification of the $\D$-unique complete $r$-partite graphs. For a given $r \in \nat$ and $a_1,\ldots,a_r \in \nat$, we let $\K(a_1,\ldots,a_r)$ denote the complete $r$-partite graph on $n=a_1 + \cdots + a_r$ vertices with color classes of size $a_1,\ldots,a_r$.

\begin{theorem}\label{thm:complete}
Let $r \in \nat$ and let $a_1,a_2,\ldots,a_r \in \nat$.  Then the complete $r$-partite graph $\K(a_1,\ldots,a_r)$ is $\ca{D}$-unique if and only if for all $1\le i < j \le r$, either $\max\{a_i,a_j\} \le 2$ or $|a_i-a_j|\ge 2$.
\end{theorem}

\noindent The proof of Theorem \ref{thm:complete} is given in Section \ref{sec:pfofcomplete}.  Our approach, which considers the structure of the family of non-dominating sets, also allows us to determine the $\D$-equivalence class of complete equipartite graphs of sufficiently high order.  We provide some additional notation to state this result. When $a_1=a_2=\cdots = a_r =: a$, we write $\K_r(a)$ in place of $\K(a_1,\ldots,a_r)$.  Note then that $\K(a)=\K_1(a)$ denotes an independent set on $a$ vertices, while the complete graph $K_r$ on $r$ vertices is thus $K_r=\K_r(1)$.  We also let $\ca{E}$ denote the empty graph (with no vertices or edges), and we define $\K_0(a)=\ca{E}$ for all $a \in \nat$.

We obtain a partial extension of Theorem \ref{thm:AAE} by determining the $\D$-equivalence class of the complete equipartite graphs $\K_r(a)$ for large $a$. To do so, we consider the {\em join} of graphs $G$ and $H$, denoted $G\vee H$, which we emphasize is defined only when $V(G) \cap V(H) = \emptyset$. In particular, for disjoint $G$ and $H$, $G \vee H$ is defined to be the graph formed by the union of $G$ and $H$ together with the edges in the complete bipartite graph with partitions $V(G)$ and $V(H)$. See the example in Figure \ref{fig:join}. (Note that $G \vee \ca{E}=G$ for all graphs $G$.)  %We emphasize that in this definition and throughout this paper, whenever we write $G \vee H$ we implicitly assume $V(G)\cap V(H)=\emptyset$.
We extend this definition in the obvious way to the {\em join} of $k$ disjoint graphs $G_1,\ldots,G_k$, denoted by $G_1 \vee G_2 \vee \cdots \vee G_k = \bigvee_{i=1}^k G_i$.  In particular, note that $\K(a_1,\ldots,a_r)=\bigvee_{i=1}^r \K(a_i)$.

\begin{figure}
\label{fig:join}

\begin{center}
\begin{center}

\begin{tikzpicture}[scale=.8]

    % Draw C_5
    \foreach \s in {0,72,144,216,288}{
        \fill (\s:1) circle (2pt);
        \draw[thick] (\s:1) -- (\s+72:1);

    }

    \node at (2,0) {$\bigvee$};

    % Draw K_3
    \fill (4,1) circle (2pt);
    \fill (4,-1) circle (2pt);
    \fill (3.2,0) circle (2pt);
    \draw[thick] (4,1) -- (4,-1) -- (3.25,0) -- (4,1);

    \node at (5.5,0) {{\large \bf =}};

    \foreach \s in {0,72,144,216,288}{
    \foreach \a in {60,180,300}{

        % Draw C_5 on the right
        \fill ($(8,0)+(\s:1)$) circle (2pt);
        \draw[thick] ($(8,0)+(\s:1)$) -- ($(8,0)+(\s+72:1)$);

        % Draw triangle on the right
        \fill (12,1) circle (2pt);
        \fill (12,-1) circle (2pt);
        \fill (11.2,0) circle (2pt);
        \draw[thick] (12,1) -- (12,-1) -- (11.25,0) -- (12,1);

        % Draw the new edges from the join

        \draw ($(8,0)+(\s:1)$) -- (12,1);
        \draw ($(8,0)+(\s:1)$) -- (12,-1);
        \draw ($(8,0)+(\s:1)$) -- (11.25,0);

    }

    }

\end{tikzpicture}
\end{center}
\end{center}

\caption{The join of $C_5$ and $K_3$.}
\end{figure}
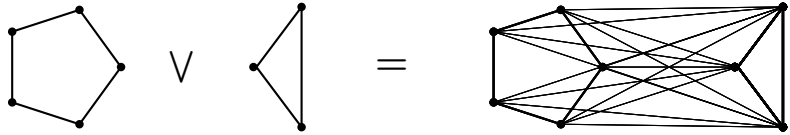

Now, for $r \ge 2$ and $0 \le t \le \lfloor r/2 \rfloor$, we define the graph $J_r(a,t)$ on $a\cdot r$ vertices as follows.  Given (vertex-disjoint) copies $H_1,\ldots,H_t$ of $H_0(a)$, let
\[J_r(a,t) = \lp \bigvee_{i=1}^t H_i \rp \vee \K_{r-2t}(a),\]
noting $J_r(a,0)=\K_r(a)$.  Examples of these graphs when $r=4$, $a=3$ and $0 \le t \le 2$ are given in Figure \ref{fig:jrat}.

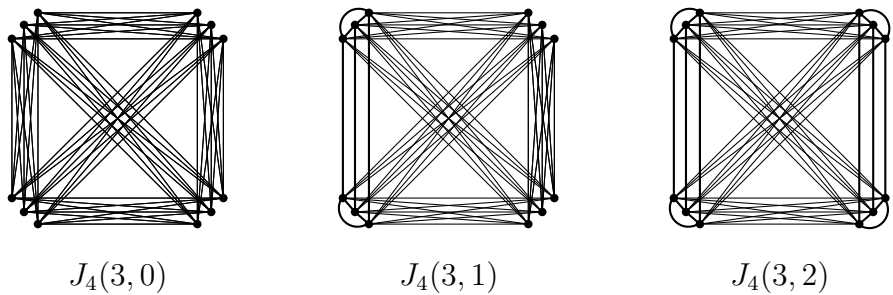
\begin{figure}

\begin{center}

\begin{tikzpicture}[scale=.8]

    % This is J_4(3,0) = T_4(12)

    % Draw T_5(15), the Turan graph on 15 vertices
    \def\r{2.2}

    % Draw the nodes
    \foreach \s in {45,135,225,315}{
        \foreach \a in {-1,0,1}{
            \fill (\s+8*\a:\r) circle (2pt);
        }
    }

    % Draw the edges

    \foreach \s in {45,135,225,315}{
        \foreach \a in {-1,0,1}{

        \foreach \t in {90,180,270}{
            \foreach \b in {-1,0,1}{
            \draw (\s+8*\a:\r) -- (\s+\t+8*\b:\r);

            }
        }

        }
    }

    \node[below] at (0,-\r) {$J_4(3,0)$};

\def\os{5.5}

    % This is J_4(3,1)

    % Draw the nodes
    \foreach \s in {45,135,225,315}{
        \foreach \a in {-1,0,1}{
            \fill ($(\os,0)+(\s+8*\a:\r)$) circle (2pt);
        }
    }

    % Draw the edges

    % First, the edges of the complete multipartite graph
    \foreach \s in {45,315}{
        \foreach \a in {-1,0,1}{
            \foreach \t in {90,180,270}{
                \foreach \b in {-1,0,1}{
                    \draw ($(\os,0)+(\s+8*\a:\r)$) -- ($(\os,0)+(\s+\t+8*\b:\r)$);
                }
            }
        }
    }

    % Next, the edges of H_0(3)
    \foreach \s in {135}{
        \foreach \a in {-1,0,1}{
            \draw[thick] ($(\os,0)+(\s + 8*\a:\r)$)-- ($(\os,0)+(\s+90 -8*\a:\r)$);
        }
    }

    \foreach \s in {135,225}{
        \foreach \a in {-1,1}{
            \draw[thick] ($(\os,0)+(\s :\r)$) -- ($(\os,0)+(\s + 8*\a:\r)$);
        }
        \draw[thick] ($(\os,0)+(\s+7.5:\r)$) arc (\s+80:\s-80:.14*\r);
    }

    \node[below] at (\os,-\r) {$J_4(3,1)$};

\def\osp{11}

    % This is J_4(3,2)

    % Draw the nodes
    \foreach \s in {45,135,225,315}{
        \foreach \a in {-1,0,1}{
            \fill ($(\osp,0)+(\s+8*\a:\r)$) circle (2pt);
        }
    }

    % Draw the edges

    % First, the edges of the complete multipartite graph
    \foreach \s in {45,315}{
        \foreach \a in {-1,0,1}{
            \foreach \t in {135,225}{
                \foreach \b in {-1,0,1}{
                    \draw ($(\osp,0)+(\s+8*\a:\r)$) -- ($(\osp,0)+(\t+8*\b:\r)$);
                }
            }
        }
    }

    % Next, the edges of H_0(3)
    \foreach \s in {135,315}{
        \foreach \a in {-1,0,1}{
            \draw[thick] ($(\osp,0)+(\s + 8*\a:\r)$)-- ($(\osp,0)+(\s+90 -8*\a:\r)$);
        }
    }

    \foreach \s in {45,135,225,315}{
        \foreach \a in {-1,1}{
            \draw[thick] ($(\osp,0)+(\s :\r)$) -- ($(\osp,0)+(\s + 8*\a:\r)$);
        }
        \draw[thick] ($(\osp,0)+(\s+7.5:\r)$) arc (\s+80:\s-80:.14*\r);
    }

    \node[below] at (\osp,-\r) {$J_4(3,2)$};

\end{tikzpicture}
\end{center}

\caption{The graphs $J_4(3,0)$,  $J_4(3,1)$, and  $J_4(3,2)$.}
\label{fig:jrat}
\end{figure}

\begin{theorem}\label{thm:balanced}
For $r \ge 2$ and $a \ge r+2$, $[\K_r(a)] = \{J_r(a,t) : 0 \le t \le \lfloor r/2 \rfloor\}$.
\end{theorem}

\noindent The proof of Theorem \ref{thm:balanced} will follow in Section \ref{sec:pfofbalanced}.

We remark that in this range of $a$, we will show that each of the graphs $J_r(a,i)$ are pairwise non-isomorphic as $i$ varies, and, in particular, the $\ca{D}$-equivalence class of $\K_r(a)$ contains $\lfloor r/2 \rfloor + 1$ distinct graphs.  We mention that our arguments also show that if $r$ is even, Theorem \ref{thm:balanced} holds for $a \ge r+1$.  In particular, Theorems \ref{thm:complete} and \ref{thm:balanced} imply Theorem \ref{thm:AAE}.

\section{Preliminaries}

For a set $X$ we let $2^X$ denote the power set of $X$, and for integers $k \ge 0$ we let $\bin{X}{k}$ denote the family of $k$-element subsets of $X$.  For a graph $G=(V,E)$, let $\N{G}\subseteq 2^V$ denote the family of non-dominating sets of $G$. Our arguments will rely in part on the following simple result.

\begin{lemma}\label{lem:nondom}
For graphs $G$ and $H$, $\N{G \vee H}=\N{G}\cup \N{H}$.
\end{lemma}

\begin{proof}
Since for each $v \in V(G)$, $N_{G \vee H}(v)\supseteq V(H)$, and since for each $w \in V(H)$, $N_{G \vee H}(w) \supseteq V(G)$, a set $X \subseteq V(G) \cup V(H)$ does not dominate $G \vee H$ if and only if either $X \subseteq V(G)$ and $X$ does not dominate $G$, or $X \subseteq V(H)$ and $X$ does not dominate $H$.
\end{proof}

Since $\K(a)$ has no proper dominating sets, the following corollary is immediate.

\begin{corollary}\label{cor:ndsofcomplete}
$\N{\K(a_1,\ldots,a_r)}$ is the set of proper subsets of the color classes of $\K(a_1,\ldots,a_r)$.
\end{corollary}

Since the family of dominating sets of $G$ is $2^{V(G)}\setminus \N{G}$, Lemma \ref{lem:nondom} also easily implies the following result from \cite{AAP}.

\begin{lemma}[\cite{AAP}, Theorem 2]\label{lem:joindompoly}
For graphs $G$ and $H$,
\begin{equation}\label{eq:joindompoly}
D(G\vee H,x) = \lp (1 + x)^{v(G)}-1\rp \lp (1 + x)^{v(H)}-1 \rp + D(G,x) + D(H,x).
\end{equation}
\end{lemma}

\noindent Note that in \eqref{eq:joindompoly}, $((1+x)^{v(G)}-1)\cdot ((1+x)^{v(H)}-1)$ is the generating function for the number of subsets $X \subseteq V(G) \cup V(H)$ such that  $X \cap V(G)\ne \emptyset$ and $X \cap V(H)\ne \emptyset$. Our arguments will only require the following corollary of Lemma \ref{lem:joindompoly}.

\begin{corollary}\label{cor:simpleobs}
For graphs $G_1$, $G_2$, and $H$, $D(G_1 \vee H,x) = D(G_2\vee H,x)$ if and only if $D(G_1,x)=D(G_2,x)$.
\end{corollary}

Next, we will appeal to the following lemma from \cite{AAP}.  Our statement differs slightly from that in \cite{AAP}, but the proof is identical; we include it for completeness.  Informally, this result shows that the domination polynomial both determines the minimum degree $\delta(G)$ of the graph and bounds below the number of vertices of minimum degree.

\begin{lemma}[\cite{AAP}, Lemma 4]\label{lem:mindeg}
Let $G=(V,E)$ be a graph of order $n$ with domination polynomial $D(G,x) = \sum_{i=1}^{n} d(G,i)x^i$, and let $\ell = \min\{j : d(G,j) = \bin{n}{j}\}$. Then $\delta(G) = n - \ell$, and for each non-dominating set $T$ of cardinality $\ell-1$, there exists a vertex $v_T$ such that $N[v_T] = V \setminus T$.
\end{lemma}

\begin{proof}
Suppose first that some vertex $v \in V$ has $d(v) < n-\ell$: then $V \setminus N[v]$ has cardinality at least $\ell$ and does not dominate, contradicting the choice of $\ell$.

Next, let $T \subseteq V$ be a non-dominating $(\ell-1)$-set: by definition, there exists a vertex $v_T \in V$ such that $N[v_T] \cap T = \emptyset$.  Since $d(v_T)\le n-(|T| +1) =n-\ell = \delta(G)$, equality holds and so $N[v_T] = V \setminus T$.
\end{proof}

\begin{corollary}[\cite{AAP}, Theorem 14]\label{cor:forcedregular}
Let $G$ be a graph, and suppose $H$ satisfies $D(H,x)=D(G,x)$.  Then $\delta(H)=\delta(G)=n-\ell$ for some $\ell \ge 1$, and $H$ contains at least $\bin{n}{\ell-1}-d(G,\ell-1)$ vertices of degree $\delta(H)$.
\end{corollary}

Finally, we will appeal to two classical results from extremal combinatorics.  The first is a special case of Tur\'an's Theorem.

\begin{theorem}[Tur\'an, \cite{T}]\label{thm:turan}
For $r \ge 3$ and $a \in \nat$, let $G$ be a graph on $n=a\cdot r$ vertices and suppose $G$ is $K_{r+1}$-free.  Then $e(G)\le e(\K_{r}(a))$, with equality if and only if $G \cong \K_{r}(a)$.
\end{theorem}

The second is a variant of the Kruskal-Katona Theorem due to Lov\'asz (\cite{L}, Problem 13.31; see also \cite{K}, Theorem 1 for a short proof).  We recall that for a family $\ca{F} \subseteq \bin{V}{k}$ of $k$-element subsets of a set $V$,  the {\em shadow} of $\ca{F}$ is the family
\[\partial \ca{F} = \left\{ Y \subseteq V : |Y| = k-1, Y \subseteq X \text{ for some } X \in \ca{F}  \right\} = \bigcup_{X \in \ca{F}} \bin{X}{k-1}.\]

\begin{theorem}[\cite{L}]\label{thm:lkkt}
Let $k \in \nat$ be an integer, $V$ be a finite set, and let $\ca{F} \subseteq \bin{V}{k}$.  If $|\ca{F}| = \bin{x}{k}$ for some real $x \ge k$, then
\[|\partial \ca{F}| \ge \bin{x}{k-1}.\]
Moreover, if equality holds then $x \in \nat$ and $\ca{F} = \bin{X}{k}$ for some $X\subseteq V$, $|X| = x$.
\end{theorem}

\section{Proof of Theorem \ref{thm:complete}}
\label{sec:pfofcomplete}

We now turn to our classification of the $\D$-unique complete $r$-partite graphs. Let $r \in \nat$ and $a_1,\ldots,a_r \in \nat$ be given.  We will show that $G=\K(a_1,\ldots,a_r)$ is $\ca{D}$-unique if and only if
\begin{equation}\label{eq:uniqcond}
\max\{a_i,a_j\} \le 2  \ \ \ \text{ or } \ \ \ |a_i-a_j| \ge 2\; ,  \ \ \ \ \forall \ 1 \le i  < j \le r.
\end{equation}

We start with the simple proof of the necessity, expanding the observations of Aalipour, Akbari and Ebrahimi mentioned in the introduction.  Suppose that \eqref{eq:uniqcond} fails for some $i<j$, noting this requires that $r \ge 2$.  Relabelling the indices if necessary, we may assume $i=1$, $j=2$, and that $a_1 \le a_2$. Since both of the conditions of the disjunction \eqref{eq:uniqcond} fail, we must have that both $a_2=\max\{a_1,a_2\} \ge 3$ and $|a_2-a_1|=a_2-a_1 \le 1$. We can thus represent the difference of $a_2$ and $a_1$ as $k:=a_2-a_1 \in \{0,1\}$.  Taking $\K(a_3,\ldots,a_r)=\ca{E}$ if $r=2$, let
\[H = H_k(a_1) \vee \K(a_3,\ldots,a_r),\]
and observe that $\K(a_1,\ldots,a_r)=\K(a_1,a_2)\vee \K(a_3,\ldots,a_r)$.  Since $a_2 = a_1 + k$, Fact \ref{fact:AAE} gives that $D(H_k(a_1),x)=D(\K(a_1,a_2),x)$, and thus Corollary \ref{cor:simpleobs} yields $D(H,x)=D(\K(a_1,\ldots,a_r),x)$.  It remains only to show that $H \not\cong \K(a_1,\ldots,a_r)$.
But this, in turn, follows as $\chi(H_k(a_1))=a_1 + k=a_2 \ge 3$, and since the chromatic number of the join of two graphs is the sum of their chromatic numbers,
\[\chi(H)=\chi(H_k(a_1)) + \chi(\K(a_3,\ldots,a_r)) \ge 3 + (r-2) > r=\chi(G),\]
and thus $H \not \cong G$.\\

To show \eqref{eq:uniqcond} suffices to ensure $\D$-uniqueness, we proceed by induction on $r$.  The base case $r=1$ is the assertion that $\K(a)$ is $\D$-unique for all $a \in \nat$. Since $\K(a)$ is an independent set on $a$ vertices, the only dominating set is its entire vertex set. From the simple observation that any graph containing an edge also contains a proper dominating set, we see that $\K(a)$ is in fact $\D$-unique for all $a \in \nat$.

We therefore continue our inductive proof by assuming $r \ge 2$ and \eqref{eq:uniqcond} holds.  To simplify our further notation, for $G = \K(a_1,\ldots,a_r)$ we relabel the indices if necessary so that
\[a_1 \le a_2 \le \cdots \le a_r.\]
Let $n = a_1 + a_2 + \cdots + a_r=v(G)$, and let $H=(V,E)$ be any graph satisfying $D(H,x)=D(G,x)$, which immediately implies that $v(H)=n$.  Our argument proceeds by considering separately the cases $a_r \le 2$ and $a_r > 2$.\\

{\bf Case 1. $a_r \le 2$.}  Since $a_r = \max_i a_i$, the color classes of $G$ must all be of size $1$ or size $2$. Letting $d = |\{i : a_i = 2\}|$ and $s=r-d$, $G$ has $s$ singleton color classes and $d$ doubleton color classes.  In particular, it follows from Corollary~\ref{cor:ndsofcomplete} that $d(G,2)=\bin{n}{2}$ and $d(G,1)=n-2d$.  Since $D(H,x)=D(G,x)$, it follows that $H$ has exactly $n-2d$ vertices of degree $n-1$ (which dominate).  Furthermore, Lemma~\ref{lem:mindeg} implies that $\delta(H)=n-2$ if $d > 0$, and thus $H$ contains exactly $2d$ vertices of degree $n-2$.  To see that this implies $H \cong G$, simply observe that the complement of each consists of $n-2d$ isolated vertices and a matching on $2d$ vertices.\\

{\bf Case 2. $a_r > 2$.}  In this case, \eqref{eq:uniqcond} implies $a_i \le a_r - 2$ for all $i < r$.  Since the only color class of $G$ containing proper subsets of cardinality at least $a_r-2$ is the unique color class of size $a_r$, Corollary~\ref{cor:ndsofcomplete} yields
\begin{align}
\label{eq:dHar} d(H,a_r)=d(G,a_r)&=\textstyle\bin{n}{a_r}, \\
\label{eq:dHarm1} d(H,a_r -1)=d(G,a_r-1) &= \textstyle\bin{n}{a_r - 1} - \textstyle\bin{a_r}{a_r -1} \,\,\, \text{ and } \\
\label{eq:dHarm2} d(H,a_r-2) =d(G,a_r-2)&= \textstyle\bin{n}{a_r -2 } - \textstyle\bin{a_r}{a_r - 2}.
\end{align}

To use these bounds, we consider the following families of non-dominating sets of $H$.
\begin{align*}
\ca{F} &= \{X \in \N{H} : |X| = a_r-1\}, & \text{ and } & & \ca{G} &= \{Y \in \N{H} : |Y| = a_r-2\}.
\end{align*}
Since $X \in \N{H}$ implies $2^X \subseteq \N{H}$, it easily follows that $\partial \ca{F} \subseteq \ca{G}$ so by \eqref{eq:dHarm1}, \eqref{eq:dHarm2} and Theorem~\ref{thm:lkkt},
\[\bin{a_r}{a_r-2} = |\ca{G}| \ge |\partial \ca{F}| \ge \bin{a_r}{(a_r-1)-1} = \bin{a_r}{a_r-2}.\]
As equality holds throughout, it follows (again by Theorem~\ref{thm:lkkt}) that $\ca{F}=\bin{A}{a_r-1}$ for some $A \subseteq V$ with $|A|=a_r$.\\

Next, by \eqref{eq:dHar} and \eqref{eq:dHarm1} we may apply Lemma~\ref{lem:mindeg} with $\ell = a_r$ to conclude that for each non-dominating set $X \in \ca{F}$, we may select a vertex $v_X$ such that $N_H[v_X] = V\setminus X$.  Let $Z=\{v_X : X \in \ca{F}\}$ be the set of these vertices, observing that
\[|Z|=|\ca{F}|=\bin{a_r}{a_r-1}=a_r=|A|.\]
Since $\ca{F}=\bin{A}{a_r-1}$, for each $v =v_X \in Z$ we have $|N[v_X] \cap A| = |A\setminus X|=1$.  This implies that that each $T \subseteq Z$ with $|T|=a_r-1$ must be a non-dominating set in $H$, and thus $\bin{Z}{a_r-1}\subseteq \ca{F} = \bin{A}{a_r-1}$, which, as $|Z| = |A|$, yields that $Z=A$.

Consequently, for each $v \in A=Z$ we have $v = v_{A \setminus \{v\}}$ and so $N_H(z) = V\setminus A$.  But this implies that $A$ is an independent set and that
\[H = H[V \setminus A] \vee H[A] \cong H[V\setminus A] \vee \K(a_r).\]
Since $G=\K(a_1,\ldots,a_{r-1}) \vee \K(a_r)$, Corollary~\ref{cor:simpleobs} implies $D(H[V\setminus A],x)=D(\K(a_1,\ldots,a_{r-1}),x)$, and then induction implies $H[V \setminus A] \cong \K(a_1,\ldots,a_{r-1})$, and so
\[H \cong \K(a_1,\ldots,a_{r-1}) \vee \K(a_r) \cong G,\]
completing the proof of Theorem~\ref{thm:complete}.

\section{Proof of Theorem \ref{thm:balanced}}
\label{sec:pfofbalanced}
We begin by fixing $r \ge 2$ and $a \ge r+2$.  Let $G = \K_r(a)$, and let $H=(V,E)$ be any graph satisfying $D(H,x)=D(G,x)$.  To prove Theorem \ref{thm:balanced} we must show that for some $0 \le t \le \lfloor r/2 \rfloor$ we have $H \cong J_r(a,t)$, where
\[J_r(a,t) = \lp \bigvee_{i=1}^t H_0(a) \rp \vee \K_{r-2t}(a).\]
We note that $D(J_r(a,t),x)=D(G,x)$ follows from Corollary \ref{cor:simpleobs} by induction, since $D(H_0(a),x)=D(\K(a,a),x)$.
We also mention that $J_r(a,t)\not\cong J_r(a,t')$ for $t \ne t'$, which can be verified by observing that the complement of $H_0(a)$ is connected for $a \ge 3$ and thus the complement of $J_r(a,t)$ consists of $r-2t$ components of size $a$ and $t$ of size $2a$ each.

Let $n = a\cdot r$, so $n= v(G)=v(H)$.  Since $\K_r(a)$ is $(n-a)$-regular and (by Corollary~\ref{cor:ndsofcomplete}) has $r\cdot \bin{a}{a-1}=r\cdot a= n$ non-dominating sets of cardinality $a-1$, Corollary~\ref{cor:forcedregular} implies $H$ is $(n-a)$-regular as well.  For each vertex $v \in V$, we let
\[Y_v = V\setminus N_H[v]\]
so $Y_v \in \N{H}$ and $|Y_v|=a-1$.

Our next step is to identify `where' in the graph $H$ the sets $Y_v$ occur, in a sense to be made clear below. To this end, we define an auxiliary graph $F$ on $V$ with edge set $\{uv : uv \text{ dominates } H\}$.  Since $a > 2$, the only pairs of vertices in $G$ which dominate are its edges, implying $e(F) = d(H,2) = d(G,2) = e(\K_r(a))$.

\begin{claim}\label{clm:kr1free}
$F$ is $K_{r+1}$-free.
\end{claim}

\begin{proof}[Proof of Claim \ref{clm:kr1free}]
Suppose to the contrary that $F$ contains a copy of $K_{r+1}$ with vertex set $\{v_1,\ldots,v_{r+1}\}$.  Since for $1 \le i < j \le r+1$, the pair $v_iv_j$ dominates $H$ if and only if $Y_{v_i} \cap Y_{v_j} = \emptyset$, we conclude
\[a\cdot r = |V| \ge \sum_{i=1}^{r+1}|Y_{v_i}| = (r+1)(a-1)=ar + [a - (r+1)],\]
a contradiction as $a \ge r+2$.
\end{proof}

Thus, by Theorem~\ref{thm:turan} we have $F \cong \K_r(a)$, so we partition $V = V_1 \cup \cdots \cup V_r$ into the $r$ color classes of the auxiliary graph $F$.  Since any subset of $V$ intersecting more than one set $V_i$ contains an edge of $F$ and hence dominates, it follows that for each $v \in V$ we must have $Y_v \subseteq V_i$ for some $i$.

Now, for $1 \le i \le r$, let $U_i = \bigcup_{v \in V_i} Y_v$.  Since $v \ne w$ implies $Y_v \ne Y_w$ and since $a \ge 2$, it follows that $|U_i| \ge a$.  The proof of Theorem~\ref{thm:balanced} is then an immediate consequence of the following claim.

\begin{claim}\label{clm:nonnbrslieinparts}
For each $1 \le i \le r$ there is a unique $j$ (possibly with $j=i$) such that $U_i = V_j$ and $U_j = V_i$, and
\[H[V_i \cup V_j] \cong \begin{cases}
\K(a) & \text{ if } i=j,\\
H_0(a) & \text{ if } i \ne j.
\end{cases}\]
\end{claim}

\begin{proof}[Proof of Claim \ref{clm:nonnbrslieinparts}]
Let $1 \le i \le r$ be given and let $v,w \in V_i$ be distinct vertices.  Let $j,k$ be given so that $Y_v \subseteq V_j$ and $Y_w \subseteq V_k$.  Since $vw$ is not an edge of $F$ and hence does not dominate $H$, $Y_v \cap Y_w \ne \emptyset$ follows, implying $k=j$ and, letting $w \in V_i\setminus \{v_j\}$ vary, that $U_i\subseteq V_j$.  Since $|V_j|=a \le |U_i|$, $U_i = V_j$ follows.

Now, if $i = j$, then for each $v \in V_i$ we have $V_i = \{v\} \cup Y_v$, implying $V_i$ is independent in $H$.  If $i \ne j$, then for each $v \in V_i$, $Y_v \subseteq V_j$ implies $v$ is adjacent to all other vertices in $V_i$ (so $H[V_i] \cong K_a$), and to exactly one vertex in $V_j$.  Since $w \in Y_v$ implies $v \in Y_w$, $U_j = V_i$ follows (which easily yields the uniqueness of $j$), and by symmetric reasoning we have $H[V_j] \cong K_a$ and each vertex in $V_j$ has exactly one neighbor in $V_i$.  Thus, the edges between $V_i$ and $V_j$ in $H$ must form a matching, and $H[V_i \cup V_j] \cong H_0(a)$ then follows.
\end{proof}

Since $U_i \cup U_j \subseteq V_i \cup V_j$ implies $H = H[V_i \cup V_j] \vee H[V \setminus (V_i \cup V_j)]$ (as the set $U_i$ contains all the non-neighbors of vertices in $V_i$), Claim \ref{clm:nonnbrslieinparts} implies $H$ is the join of $t$ copies of $H_0(a)$, where $0 \le t \le \lfloor r/2 \rfloor$ since each copy has $2a$ vertices, and $r-2t$ copies of $\K(a)$.  Since $\K_{r-2t}(a)$ is the join of $(r-2t)$ copies of $\K(a)$, $H \cong J_r(a,t)$, and Theorem~\ref{thm:balanced} follows.

\section{Conclusion}

Our results have settled in the affirmative an open conjecture that complete bipartite graphs where the size of the partitions differs by at least two are $\D$-unique; that is, any two complete bipartite graphs where the size differs by more than one that have the same domination polynomial are isomorphic. Furthermore, we have extended these arguments to obtain a complete classification of $\D$-unique $r$-partite graphs, as well as the $\ca{D}$-equivalence class of complete equipartite graphs.  However, our approach has shed little light on how to prove the conjecture (mentioned in Section \ref{sec:intro} after Theorem \ref{thm:AAE}) characterizing the $\ca{D}$-equivalence class of $K_{n,n+1}$, and we feel this is the next problem in this area worthy of attention.

\section{Acknowledgements}
We would like to thank Ghodratollah Aalipour for providing us with preprints of \cite{AAE} and \cite{AAE2}.

\end{document}